\documentclass{amsart}

\usepackage{amsmath,amssymb,verbatim}
\usepackage{amsthm}
\usepackage{enumerate}
\usepackage{url}

\newtheorem{theorem}{Theorem}[section]
\newtheorem{fact}[theorem]{Fact}

\newtheorem{corollary}[theorem]{Corollary}
\newtheorem{proposition}[theorem]{Proposition}

\theoremstyle{definition}

\newtheorem{remark}[theorem]{Remark}

\newcommand{\abar}{\bar{a}}
\newcommand{\bbar}{\bar{b}}
\newcommand{\cbar}{\bar{c}}

\newcommand{\ybar}{\bar{y}}
\newcommand{\zbar}{\bar{z}}

\def\seq{\subseteq}
\def\inv{^{\text{-}1}}
\def\F{\mathbb{F}}
\def\smd{\raisebox{.4pt}{\textrm{\scriptsize{~\!$\triangle$\!~}}}}
\def\FER{\operatorname{FER}}
\def\Def{\operatorname{Def}}

\def\N{\mathbb{N}}
\def\Z{\mathbb{Z}}
\def\R{\mathbb{R}}

\def\cL{\mathcal{L}}
\def\cU{\mathcal{U}}
\def\cP{\mathcal{P}}

\def\st{\operatorname{st}}
\def\deg{\operatorname{deg}}

\title{A group version of  stable regularity}

\author[G. Conant]{G. Conant}

\address{Department of Mathematics\\
University of Notre Dame\\
Notre Dame, IN, 46656\\
 USA}
\email{gconant@nd.edu}

\author[A. Pillay]{A. Pillay}
\thanks{The second author was supported by NSF grants DMS-1360702 and DMS-1665035.}

\address{Department of Mathematics\\
University of Notre Dame\\
Notre Dame, IN, 46656\\
 USA}
\email{apillay@nd.edu}

\author[C. Terry]{C. Terry}

\address{Department of Mathematics\\
University of Maryland\\
College Park, MD 20742\\
 USA}
\email{cterry@umd.edu}

\date{September 30, 2018}

\begin{document}

\begin{abstract}
We prove that, given $\epsilon>0$ and  $k\geq 1$, there is an integer $n$ such that the following holds. Suppose $G$ is a  finite group and $A\subseteq G$ is $k$-stable. Then there is a normal subgroup $H\leq G$ of index at most $n$, and a set $Y\subseteq G$, which is a union of cosets of $H$, such that $|A\smd Y|\leq\epsilon|H|$. It follows that, for any  coset $C$ of $H$, either $|C\cap A|\leq \epsilon|H|$ or $|C\setminus A|\leq \epsilon|H|$. This qualitatively generalizes recent work of Terry and Wolf on vector spaces over $\mathbb{F}_p$.
\end{abstract}

\maketitle

\section{Introduction}

Given a group $G$, a subset $A\seq G$, and an integer $k\geq 1$, we say $A$ is \emph{$k$-stable} if there do not exist $a_1,\ldots, a_k,b_1,\ldots, b_k\in G$ such that $a_ib_j\in A$ if and only if $i\leq j$.  In  \cite{TW}, Terry and Wolf prove the following arithmetic regularity lemma for $k$-stable subsets $A\seq\F_p^n$, where $\mathbb{F}_p^n$ is the additive group of the $n$-dimensional vector space over a fixed field of prime order $p$. 

\begin{theorem}[Terry \& Wolf \cite{TW}]\label{thm:TW}
For any $k\geq 1$, $\epsilon>0$, and prime $p$, there is $N=N(k,\epsilon,p)$ such that the following holds for any $n\geq N$. Suppose $G:=\F_p^n$ and that $A\seq G$ is $k$-stable. Then there is a subgroup $H\leq G$ of index at most $p^{\epsilon^{\text{-}O_k(1)}}$ such that for any $g\in G$, either $|H\cap (A-g)|\leq\epsilon|H|$ or $|H\setminus (A-g)|\leq\epsilon|H|$. 
\end{theorem}

Theorem \ref{thm:TW} is a strengthened version of Green's ``Regularity lemma in $({\mathbb{Z}}/2{\mathbb{Z}})^{n}$", Theorem 2.1 of \cite{Green}, for stable subsets of $\F_p^n$.  In this paper, we use model theoretic techniques to generalize Theorem \ref{thm:TW} to arbitrary finite (not necessarily abelian) groups $G$, but without the quantitative information on the index of $H$ in $G$. Specifically, we prove:

\begin{theorem}\label{thm:main}
For any $k\geq 1$ and $\epsilon>0$, there is $n=n(k,\epsilon)$ such that the following holds. Suppose $G$ is a finite group and $A\seq G$ is $k$-stable. Then there is a normal subgroup $H\leq G$, of index at most $n$, such that for each coset $C$ of $H$ in $G$ either $|C\cap A|\leq\epsilon|H|$ or $|C\setminus A|\leq \epsilon|H|$. 
\end{theorem}

Theorem \ref{thm:main} is connected to the stable regularity theorem for graphs.  When $G$ is abelian (and written additively), the relation $\delta(x,y)$ defined by $x+y\in A$ induces a graph on $G$. Then the statement of Theorem 1.2, restricted to such $(G,A)$ with $A$ $k$-stable, yields the stable graph regularity lemma of \cite{MS}  (and \cite{MP}),  with the further property that the pieces of the regular partition are cosets of a given subgroup (but without the quantitative aspects of \cite{MS}). For general groups, it is more natural to consider bipartite graphs and, in this case, the normality of the subgroup $H$ in Theorem \ref{thm:main} allows us to deduce a similar regularity lemma. Details are given in Corollary \ref{cor:reg}.  Hence we view Theorem 1.2 as an algebraic regularity statement for stable subsets of groups, in a rather general environment (compared to the setting of \cite{Green} and \cite{TW}). 

Theorem \ref{thm:main} can also be viewed as a structure theorem for subsets $A$ of arbitrary finite groups $G$ which are {\em stable} in the sense that the relation $xy\in A$ is $k$-stable in $G$ for some fixed $k$. In \cite{TW}, the regularity statement of Theorem \ref{thm:TW} is used to obtain a structural result on $k$-stable subsets of $\F_p^n$ in terms of cosets of the subgroup $H$. In particular, with $G$, $H$, and $A$ as in Theorem \ref{thm:TW}, there is a subset $Y\seq G$, which is a union of cosets of $H$, such that $|A\smd Y|\leq\epsilon|G|$ (see \cite[Corollary 1]{TW}).  In fact, we will prove Theorem \ref{thm:main} by proving this structure result for arbitrary finite groups first, and moreover with $\epsilon|G|$ strengthened to $\epsilon|H|$ (but, again, without quantitative bounds on the index of $H$). 

\begin{theorem}\label{thm:structure}
For any $k\geq 1$ and $\epsilon>0$, there is $n=n(k,\epsilon)$  such that the following holds. Suppose $G$ is a finite group  and $A\seq G$ is $k$-stable. Then there is a normal subgroup $H\leq G$, of index at most $n$, and a subset $Y\seq G$, which is a union of cosets of $H$, such that $|A\smd Y|\leq\epsilon|H|$. 
\end{theorem}

Our main theorems will be proved by referring to  results in ``local" stability and stable group theory, and using an ultraproduct construction.  In Theorem \ref{main-lemma} we take the opportunity to give a clear statement of what stable group theoretic results holds in the local (formula-by-formula) setting. While local stability theory is very well developed for arbitrary structures (see, e.g., \cite{H}, \cite{HP}, \cite[Section II.2]{shelah}), a precise account for groups, involving local connected components and local Keisler measures, is not available in the literature.  It will be convenient to refer to \cite{HP} where the  results are stated in a form suitable for adaptation to the current paper, but there are other sources such as \cite{pillay} and \cite{H}. 

Our methods are analogous to the proof  in \cite{MP} of the stable regularity lemma (for graphs) \cite{MS}, which again used local stability theory and ultraproducts, without giving explicit bounds.  However, whereas \cite{MP} gave a new proof of known results, in this paper we prove genuinely new results.  Specifically, the statements of Theorems \ref{thm:main} and \ref{thm:structure} for arbitrary groups (even for abelian groups) are entirely new. Moreover, our work provides concrete quantitative improvements  (namely strengthening $\epsilon|G|$ to $\epsilon|H|$ in Theorem \ref{thm:structure}), for which a finitary argument has not been found.    Ongoing work of Wolf and the third author suggests that generalizing the original proof of Theorem \ref{thm:TW} even to the setting of abelian groups will require non-trivial adaptations.  For instance, in that setting, subgroups are replaced by certain approximate subgroups called Bohr sets. Therefore, the infinitary perspective and model theoretic tools employed in this paper allow us to prove facts which seem difficult to prove fully in the finite setting.  On the other hand, as Theorems \ref{thm:main} and \ref{thm:structure} do not obtain explicit bounds on $n(k,\epsilon)$, finding finitary proofs which yield such bounds remains an interesting question.

In Section 2, we obtain the necessary results in local stable group theory. In Section 3 we give the proofs of Theorems \ref{thm:main} and \ref{thm:structure}.

\section{Local stable group theory}

Stability theory is typically about definability in stable theories, and stable group theory is about stability in the presence of a definable group operation (or transitive action). It is important and useful to know that much of the machinery of stability is available when a given formula $\delta(x,\ybar)$ is stable (see the definition below) and we work inside the Boolean closure of sets defined by instances of $\delta$, even though the ambient first order theory may be unstable. We call this ``local" stability, and we give details of the group version in this section. 

Throughout this section, we work with a  fixed group $G$ (possibly infinite) and, given $a,b\in G$, we write $ab$ for the product of $a$ and $b$ in $G$. For this section, we also fix a relation $\delta(x,\ybar)$ on $G$, where $\ybar$ is some finite tuple of variables. We assume $\delta$ is \emph{(left) invariant}, i.e.  for all $\bar{b}\in G^{|\ybar|}$ and $a\in G$, there is some $\bar{c}\in G^{|\ybar|}$ such that $\delta(ax,\bar{b})$ and $\delta(x,\bar{c})$ are equivalent. 

Let $M_\delta$ be the first-order structure which expands the group structure on $G$ by the relation $\delta$. We assume  that $\delta$ is \emph{stable}, i.e. for some $k\geq 1$, there do not exist $a_1,\ldots,a_k\in G$ and $\bbar_1,\ldots,\bbar_k$ in $G^{|\ybar|}$ such that $\delta(a_i,\bbar_j)$ holds if and only if $i\leq j$ (in this case we also say that $\delta$ is \emph{$k$-stable}).

A \emph{$\delta$-formula} is a formula $\phi(x)$ given by  a finite Boolean combination of instances of $\delta(x,\bar{b})$ for $\bar{b}\in G^{|\ybar|}$. We treat $x=x$ and $x\neq x$ as degenerate $\delta$-formulas. Let $S_\delta(M_\delta)$ denote the space of \emph{complete $\delta$-types} over $G$, i.e., maximal sets of $\delta$-formulas such that any finite subset is simultaneously satisfied by an element of $G$. Recall that $S_\delta(M_\delta)$ is a totally disconnected compact Hausdorff space under the topology whose basic open sets are $\{p\in S_\delta(M_\delta):\phi(x)\in p\}$ for some $\delta$-formula $\phi(x)$.  We will often conflate definable sets with the formulas defining them, for instance by saying a type $p\in S_{\delta}(M_{\delta})$ \emph{contains} a $\delta$-definable set $X$ if $p$ contains the $\delta$-formula defining $X$.

A set $A\seq G$ is \emph{(left) generic} if finitely many left translates of $A$ cover $G$. A $\delta$-formula $\phi(x)$ is \emph{generic} if it defines a generic subset of $G$. A type $p(x)\in S_\delta(M_\delta)$ is \emph{generic} if every formula in $p(x)$ is generic. Note also that we have a left action of $G$ on $S_\delta(M_\delta)$: given $p\in S_\delta(M_\delta)$ and $g\in G$, $g.p=\{\phi(g\inv x):\phi(x)\in p\}$. In the following, we always work with the left action of $G$ on definable sets and types. In particular, when we say ``translate" and ``coset", we always mean ``left translate" and ``left coset", etc.

Let $\FER_\delta$ denote the collection of $\emptyset$-definable equivalence relations $E(x_1,x_2)$ on $G$ such that $E$ has finitely many classes and each $E$-class is defined by a $\delta$-formula.  An equivalence relation $E\in\FER_{\delta}$ is \emph{$G$-invariant} if, for any $a,b,g\in G$, $E(a,b)$ holds if and only if $E(ga,gb)$ holds. 

The following fact is taken from \cite{HP}, which is written in the setting of homogeneous spaces $(G,S)$. Our setting is the special case when $G=S$. 

\begin{fact}\label{fact:HP}
Let $X$ be the set of generic types $p\in S_\delta(M_\delta)$.
\begin{enumerate}[$(i)$]
\item $X$ is finite and nonempty.
\item $G$ acts transitively on $X$.
\item There is $E_\delta\in \FER_\delta$ such that $E_\delta$ is $G$-invariant and, for $p_1,p_2\in X$, $p_1=p_2$ if and only if $p_1$ and $p_2$ contain the same $E_\delta$-class.
\item A $\delta$-formula $\phi(x)$ is generic if and only if it is contained in some $p\in X$.
\end{enumerate}
\end{fact}
\begin{proof}
Parts $(i)$, $(ii)$, and $(iii)$ are precisely Lemma 5.16 of \cite{HP}. Part $(iv)$ is given as a claim in the proof of this lemma. 
\end{proof}

By Remark 5.17$(i)$ of \cite{HP}, a $\delta$-formula $\phi(x)$ is generic if and only if, for all $g\in G$, the formula $\phi(gx)$ does not fork over $\emptyset$ (equivalently, does not divide over $\emptyset$), i.e., in modern jargon, $\phi(x)$ is \emph{$f$-generic}. Indeed, the following is evident from the proof of  Remark 5.17 in \cite{HP}.

\begin{remark}\label{rem:fgen}
A $\delta$-formula $\phi(x)$ is generic if and only if for every indiscernible sequence $(g_i)_{i\in\omega}$ in a sufficiently saturated extension $M^*_\delta\succ M_\delta$, $\{\phi(g_ix):i\in\omega\}$ is consistent. 
\end{remark}

Let $\Def_{\delta}(G)$ be the collection of $\delta$-formulas, which we identify with the Boolean algebra of subsets of $G$ defined by $\delta$-formulas. A \emph{$\delta$-Keisler measure} is a  finitely additive probability measure on $\Def_\delta(G)$. Let $G^0_\delta$ denote the $E_\delta$-class of the identity in $G$ (where $E_\delta$ is as in Fact \ref{fact:HP}).

The following theorem summarizes several main features of stable group theory in this local setting. Under the global assumption of stability, this is well-known and standard in any text on stability theory (e.g., \cite{pillay}). However, a clear account of the following result is not available in the literature, and so we take the opportunity to show precisely what works in the local setting, as well as some subtle differences (see Remark \ref{rem:local}).

\begin{theorem}\label{main-lemma}
Fix a group $G$, and let $\delta(x,\bar{y})$ be a stable invariant relation on $G$. 
\begin{enumerate}[$(i)$]
\item $G^0_\delta$ is a subgroup of $G$ of finite index, and is in $\Def_\delta(G)$. The $E_\delta$-classes in $G$ are precisely the left cosets of $G^0_\delta$.
\item Each left coset of $G^0_\delta$ is contained in a unique generic type $p\in S_\delta(M_\delta)$.
\item For any $\phi(x)\in \Def_\delta(G)$ and any left coset $C$ of $G^{0}_{\delta}$ in $G$,  exactly one of $C\cap\phi(G)$ or $C\setminus\phi(G)$ is generic.  

\item For any $\phi(x)\in\Def_\delta(G)$, if $Y$ is the union of the cosets of $G^0_\delta$ whose intersection with $\phi(G)$ is generic, then $\phi(G)\smd Y$ is not generic.

\item $G^0_\delta$ is the smallest finite-index subgroup of $G$ in $\Def_\delta(G)$.
\item There is a unique left-invariant $\delta$-Keisler measure $\mu_\delta$ on $\Def_\delta(G)$. Moreover, given $\phi(x)\in\Def_\delta(G)$, $\mu_\delta(\phi(x))>0$ if and only if $\phi(x)$ is generic.
\end{enumerate}
\end{theorem}
\begin{proof}
Part $(i)$. The fact that $G^0_\delta$ is a subgroup of $G$, and that its cosets are precisely the $E_\delta$-classes of $G$, follows from $G$-invariance of $E_\delta$. Since $E_\delta$ has only finitely  many classes, $G^0_\delta$ has finite index. 

Part $(ii)$. Since $G^0_\delta$ has finite index, it is generic. Thus every coset of $G^0_\delta$ is contained in some generic type by Fact \ref{fact:HP}$(ii)$, which is unique by Fact \ref{fact:HP}$(iii)$. 

Part $(iii)$.  Let $p\in S_{\delta}(M_{\delta})$ be the unique generic type containing the coset $C$ (by part $(ii)$).  By parts $(iii)$ and $(iv)$ of Fact \ref{fact:HP}, a $\delta$-definable subset of $C$ is generic if and only if the $\delta$-formula defining it is in $p$. As exactly one of $\phi(x)$ or $\neg\phi(x)$ is in $p$, we see that exactly one of $C\cap\phi(G)$ or $C\setminus\phi(G)$ is generic.

Part $(iv)$. Since the nongeneric definable sets form an ideal (by parts $(i)$ and $(iv)$ of Fact \ref{fact:HP}), this is simply a reformulation of part $(iii)$. 

Part $(v)$. If this fails, then there is a $\delta$-definable proper subgroup $H\leq G^0_\delta$, which has finite index in $G$. If $\phi(x)\in\Def_\delta(G)$ defines $H$, then the set $Y$ from part $(iv)$ is just $G^0_{\delta}$, and $\phi(G)\smd Y=G^0_\delta\backslash H$. Now $\phi(G)\smd Y$ is nonempty union of  left cosets of $H$, and so is generic (as $H$ as finite index), which is a contradiction.

Part $(vi)$. We first claim that if $\mu$ is a left-invariant $\delta$-Keisler measure on $\Def_\delta(G)$, and $\phi(x)$ is a $\delta$-formula, then $\mu(\phi(x))>0$ if and only if $\phi(x)$ is generic. To see this, first note that the right-to-left direction is immediate from invariance and finite additivity. So suppose $\mu(\phi(x))>0$. Since genericity of $\phi(x)$ is preserved when in elementary extensions and substructures, and $\mu$ can be lifted to an elementary extension (e.g., following \cite[Section 2]{HPP}), we may assume without loss of generality that $M_\delta$ is sufficiently saturated. Now $\phi(x)$ is generic by Remark \ref{rem:fgen} and a standard exercise on probability measures (see, e.g., \cite[Lemma 2.2]{NP}). 

Let $n$ be the index of $G^0_\delta$ in $G$.  Suppose $\mu$ is a left-invariant $\delta$-Keisler measure on $\Def_\delta(G)$. Note that any coset of $G^0_\delta$ has $\mu$-measure $\frac{1}{n}$. Fix $\phi(x)\in\Def_\delta(G)$, and let $Y=C_1\cup\ldots\cup C_k$ be as in part $(iv)$. Since $\phi(G)\smd Y$ is not generic, we have $\mu(\phi(x))=\mu(Y)=\frac{k}{n}$. This shows $\mu$ is uniquely defined, and one checks that $\mu_\delta\colon \phi(x)\mapsto \frac{k}{n}$ is a left-invariant $\delta$-Keisler measure  on $\Def_\delta(G)$.  In particular, finite additivity follows from part $(iii)$ and the fact that the non-generic sets form an ideal (by parts $(i)$ and $(iv)$ of Fact \ref{fact:HP}).
\end{proof}

\begin{remark}\label{rem:local}
The group $G^0_\delta$ is obviously a local analog of the \emph{connected component} $G^0$ of a  stable group $G$ (i.e., $G^0$ is the intersection of all definable subgroups of finite index). However, unlike the global case, there is no reason that $G^0_{\delta}$ should be a \emph{normal} subgroup of $G$, as a conjugate $aG^0_\delta a\inv$ is not necessarily $\delta$-definable. Along these same lines, the unique left-invariant Keisler measure on $\Def_\delta(G)$ is not defined on right translates of $\delta$-definable sets, which is again in contrast to the global setting where the unique left-invariant Keisler measure on a stable group is also the unique right-invariant Keisler measure. 
\end{remark}

\begin{remark}\label{rem:indexbound}
In Theorem \ref{main-lemma}, it is natural to ask about how the index of $G^0_\delta$ depends on $k$, where $k\geq 1$ is such that $\delta(x;\ybar)$ is $k$-stable. We observe that there is no direct relationship in general. For example, let $G=(\Z,+)$ and let $\delta(x,y)$ be $x+y\in n\Z$, for some fixed $n\geq 1$. Then $\delta(x,y)$ is $2$-stable since $n\Z$ is a subgroup. But $G^0_\delta\seq n\Z$ and so $[G:G^0_\delta]\geq n$ (in fact $G^0_\delta=n\Z$). Conversely, given $k\geq 1$ let $\delta(x,y)$ be $x+y\in\{0,\ldots,k-1\}$. Then $\delta(x,y)$ is $(k+1)$-stable, but not $k$-stable. Moreover, any $\delta$-formula defines a finite or cofinite subset of $\Z$, and so $G^0_\delta$ must be $\Z$.
\end{remark}

\section{Proof of the main results}

We can now prove Theorems \ref{thm:main} and \ref{thm:structure} from Section 1. 

\begin{proof}[Proof of Theorem \ref{thm:structure}] 
For a contradiction, suppose that the theorem is false.  Then there is $k>0$ and $\epsilon >0$ such that the following hold. For each $i\geq 0$, there is a finite group $G_i$ and a $k$-stable subset $A_i\subseteq G_i$ such that for any normal $H\leq G_i$ of index at most $i$, and any $Y\seq G_i$, which is a union of cosets of $H$, 
\[
|A_i\smd Y|>\epsilon|H|.
\]
Note, in particular, that $|G_i|>i$, since otherwise we could take $H$ to be the trivial group and contradict the assumptions on $G_i$.

The strategy of the proof is to work with an ultraproduct of the structures $(G_i,A_i)$, considered in the group language with an extra unary predicate, and obtain an infinite group contradicting Theorem \ref{main-lemma}. In order to apply arguments involving Keisler measures, we will need to work in a larger language containing extra symbols for measuring formulas. There are several accounts of this kind of formalism in the literature, and we will loosely follow \cite[Section 2.6]{HruAG}. 

Let $\cL$ be the language with three sorts, $S_{Gr}$, $S_P$, and $S_R$, a unary predicate $A(x)$ on $S_{Gr}$, binary functions $ \cdot\colon S_{Gr}^2\to S_{Gr}$ and $+\colon S_R^2\rightarrow S_R$, binary relations $\in$ on $S_{Gr}\times S_P$ and $<$ on $S_R\times S_R$, and a function symbol $\mu\colon S_P\rightarrow S_R$. Given an $\cL$-structure $M$ and a sort $S$ in $\cL$, we let $S^M$ denote the interpretation of $S$ in $M$. For each $i$, let $M_i$ be the $\cL$-structure with $S^{M_i}_{Gr}=(G_i,\cdot, A_i)$, $S^{M_i}_R=([0,1],+,<)$, and $S^{M_i}_P=\cP(G_i)$, and with $\in$  interpreted as the set membership relation and $\mu$ interpreted as the normalized counting measure $\mu_i\colon \cP(G_i)\rightarrow [0,1]$. Now let $\cU$ be a non-principal ultrafilter on $\N$ and consider the $\cL$-structure $M=\prod_{\cU} M_i$.  Let $G$ denote the underlying group of $M$ and let $\delta(x,y)=A(y\cdot x)$. Note that $\delta(x,y)$ is invariant and $k$-stable. 

For each $i$, let $\Def^*_\delta(G_i)$ denote the collection of all formulas in the language $\{\cdot,A\}$, with parameters from $G_i$, in a single variable $x$ (so $\Def_\delta(G_i)\seq\Def^*_\delta(G_i))$. Similarly define $\Def^*_\delta(G)$. 

Next, we will observe that, for any $\phi(x)\in \Def^*_\delta(G)$, the set $\phi(G)$ defined by $\phi(x)$ is identified with a unique element of $S^{M}_P$ via the interpretation of $\in$ in $M$.  Essentially this is because the same property is both true in each $M_i$ and uniformly expressible in $\cL$, and thus transfers via {\L}o{\'{s}}'s Theorem. Precisely, for each $G_i$ and $\phi(x,\abar)\in \Def^*_\delta(G_i)$, there is a unique element $X_{\phi(x,\abar)}$ of $S_P^{M_i}=\cP(G_i)$ such that 
$$
M_i\models \forall x (x\in X_{\phi(x,\abar)}\leftrightarrow \phi(x,\abar)).
$$
Now let $Z=X_{\phi(x,\zbar)}$ be the formula $\forall x (x\in Z\leftrightarrow \phi(x,\zbar))$, in the free variables $Z$ (of sort $S_P$) and $\zbar$ (of sort $S_{Gr}$), and note this uniformly defines $X_{\phi(x,\abar)}$ in $G_i$ for each $i$.  So for each $\phi(x,\abar)\in \Def^*_\delta(G)$, there is a unique element $X_{\phi(x,\abar)}$ in $S_P^{M}$ such that 
$$
M\models \forall x (x\in X_{\phi(x,\abar)}\leftrightarrow \phi(x,\abar)),
$$
and moreover, if $\abar=[(\abar_i)_{i\in\N}]$, then we must have that $X_{\phi(x,\abar)}=\prod_{\cU} X_{\phi(x,\abar_i)}$.

Let $\st\colon S^M_R\to[0,1]$ be the standard part map. For each $\phi(x,\abar)\in \Def^*_\delta(G)$, set $\nu(\phi(x,\abar))=\st\mu^M(X_{\phi(x,\abar)})$. It is routine to show that, for any $\phi(x,\abar)\in\Def^*_\delta(G)$, if $(\abar_i)_{i\in\N}$ is a representative for $\abar$, then
\[
\nu(\phi(x,\abar))=\lim_{\cU}\mu_i(\phi(x,\abar_i)).
\]
Thus $\nu$ is an ultralimit of counting measures, which are left-invariant finitely additive probability measures. By  \L o\'{s}'s Theorem, $\nu$ is a left-invariant $\delta$-Keisler measure on $\Def_\delta(G)$. Because $\delta$ is $k$-stable, we see that $\nu$ is the unique left-invariant probability measure on $\Def_{\delta}(G)$, given by Theorem \ref{main-lemma}. Moreover, we have the finite index subgroup $K:=G^0_\delta$ of $G$, given by Theorem \ref{main-lemma}.

Let $m$ be the index of $K$, and let $H=\bigcap_{g\in G}gKg\inv$. Then $H$ is normal of index $n\leq m!$. Moreover, $H\in\Def_\delta^*(G)$ since it is a finite intersection of conjugates of $K\in\Def_\delta(G)$. Let $H=\phi(G,\abar)$ where $\phi(x,\abar)\in \Def^*_\delta(G)$ and $\abar$ is a tuple from $M$. Let $\abar=[(\abar_i)_{i\in \mathbb{N}}]$, where each $\abar_i$ is a tuple from $M_i$.  For $i\in\N$, let $H_i=\phi(G,\abar_i)$. Define 
\[
V=\{i\geq n:\text{$H_i$ is a normal subgroup of $G_i$ of index $n$}\},
\]
and note that $V\in \cU$.  By our choice of the $G_i$, it follows that for any $i\in V$ and any $Y\seq G_i$, which is a union of cosets of $H_i$,
\[
\textstyle \mu_i(A_i \smd Y)>\frac{\epsilon}{n}.
\]
Given $I\seq [n]$, let $\ybar_I=(y_j)_{j\in I}$, and define the formula 
\[
\theta_I(x;\ybar_I,\zbar):=A(x)\smd\bigvee_{j\in I}\phi(y_j\cdot x,\zbar).
\]
Now define the formula
\begin{multline*}
\psi(\zbar):= \forall y_1\ldots \forall y_n\left(\left(\forall w \bigvee_{j=1}^n \phi(y_j\cdot w,\zbar)\right)\rightarrow \right.\\
\left. \bigwedge_{I\subseteq [n]}\forall Z \bigg(Z=X_{\theta_I(x;\ybar_I,\zbar)}\rightarrow \mu(Z)>\frac{\epsilon}{n}\bigg)\right).
\end{multline*}
In other words, given $i\geq 0$ and $\cbar\in G^{|\zbar|}$, $M_i\models \psi(\cbar)$ if and only if for any $b_1,\ldots,b_n\in G_i$, if $G_i$ is covered by the translates $b_1\phi(G_i,\cbar),\ldots,b_n\phi(G_i,\cbar)$ then, for any $I\seq [n]$,  $\mu_i(A_i\smd Y)>\frac{\epsilon}{n}$ where $Y$ is the union of $b_j\varphi(G_i,\cbar)$ for $j\in I$. In particular, $M_i\models\psi(\abar_i)$ for all $i\in V$. By \L o\'{s}'s Theorem, $M\models\psi(\abar)$. 

Let $A_*\seq G$ be the interpretation of $A(x)$ in $M$. Since $M\models\psi(\abar)$, it follows that for any $Y\seq G$, which is a union of cosets of $H$, we have $\mu^M(A_*\smd Y)>\frac{\epsilon}{n}$, and so $\nu(A_*\smd Y)\geq \frac{\epsilon}{n}>0$. Since $K$ is a union of cosets of $H$, the same statement is true where $Y$ is a union of cosets of $K$. But since $A_*$ is defined in $M$ by $\delta(x,1)\in\Def_\delta(G)$, this contradicts that $\nu$ is the measure $\mu_\delta$ from Theorem \ref{main-lemma}.
\end{proof}

\begin{proof}[Proof of Theorem \ref{thm:main}]
Given $k\geq 1$ and $\epsilon>0$, let $n=n(k,\epsilon)$  be as in Theorem \ref{thm:structure}. Suppose $G$ is a finite group and $A\seq G$ is $k$-stable. Let $H\leq G$ be the normal subgroup of index at most $n$ given by Theorem \ref{thm:structure}, and let $Y\seq G$ be a union of cosets of $H$ such that $|A\smd Y|\leq\epsilon |H|$. Let $C$ be a coset of $H$ in $G$. If $C\seq Y$ then $(C\setminus A)\seq (Y\setminus A)\seq A\smd Y$. Otherwise, since $Y$ is a union of cosets of $H$, we have $C\cap Y=\emptyset$ and so $C\cap A\seq A\setminus Y\seq A\smd Y$. 
\end{proof}

Before moving to graph regularity, we make several remarks about the previous proofs.

\begin{remark}
Given $G$, $A$, and $H$ as in the statement of Theorem \ref{thm:main} or  \ref{thm:structure}, it follows from the proof of Theorem \ref{thm:structure} that $H$ is definable from $A$ in the following sense: $H$ is a finite intersection of conjugates of a subgroup $K$, which is in the Boolean algebra of subsets of $G$ generated by left translates of $A$.
\end{remark}

\begin{remark}
As mentioned in the introduction, Theorem \ref{thm:structure} improves the structural components of the work in \cite{TW} (valid for $G=\F_p^n$), in particular strengthening $|A\smd Y|\leq\epsilon|G|$ to $|A\smd Y|\leq\epsilon|H|$. The underlying reason for this improvement is that, in the proof of Theorem \ref{thm:main}, we obtain $\nu(A_*\smd Y)=0$ in the ultraproduct. Indeed, one could reformulate the statements of Theorems \ref{thm:main} and \ref{thm:structure} to obtain any uniformly defined error (at the cost of changing the bounds).  For example, given $k\geq 1$ and $\gamma\colon \Z^+\to\R^+$, there is $n=n(k,\gamma)$ such that the following holds. Suppose $G$ is a finite group and $A\seq G$ is $k$-stable. Then there is a normal subgroup $H\leq G$, of index $m\leq n$, and a subset $Y\seq G$, which is a union of cosets of $H$, such that $|A\smd Y|\leq\gamma(m)|H|$.
\end{remark}

\begin{remark}
Continuing the thread of Remark \ref{rem:indexbound}, we observe that in Theorem \ref{thm:structure}, it is not possible to remove the dependency of $n$ on $\epsilon$. Precisely, for any $k\geq 2$ and $n\geq 1$, there is some $\epsilon>0$ such that the following holds. For any $N\geq 1$, there is a finite group $G$, with $|G|\geq N$, and a $k$-stable subset $A\seq G$ such that, for any subgroup $H\leq G$ of index at most $n$ and any set $Y\seq G$, which is a union of cosets of $H$, one has $|A\smd Y|>\epsilon|G|$.

To see this, fix $k\geq 2$ and $n\geq 1$. Pick $\epsilon>0$ small enough so that there is a prime $p>n$ satisfying $\frac{1}{1-\epsilon}<p<\frac{1}{\epsilon}$. Now fix $N\geq 1$, and find a prime $q>n$ such that $pq\geq N$. Let $G=(\Z/ p\Z)\times (\Z/ q\Z)$ and let $A=\{0\}\times (\Z/ q\Z)$. Note that $A$ is $2$-stable since it is a subgroup. Suppose $H\leq G$ is a subgroup of index at most $n$. Since $p$ and $q$ are primes greater than $n$, it follows that $H$ has index $1$, and so $H=G$. Now suppose $Y$ is a union of cosets of $H$. Then either $Y=\emptyset$ or $Y=G$, and so we have $|A\smd Y|=q$ or $|A\smd Y|=pq-q$, respectively. Since $\frac{1}{1-\epsilon}<p<\frac{1}{\epsilon}$, this yields $|A\smd Y|>\epsilon pq=\epsilon|G|$ in either case.
\end{remark} 

Finally, we deduce a graph regularity statement from Theorem \ref{thm:main}. Since we work with possibly nonabelian groups, it is more natural to consider bipartite graphs. Given a finite bipartite graph $\Gamma=(V,W;E)$, subsets $X\seq V$ and $Y\seq W$, and vertices $v\in V$ and $w\in W$, define 
\[
\deg_\Gamma(v,Y)=|\{y\in Y:E(v,y)\}|\makebox[.5in]{and}\deg_\Gamma(X,w)=|\{x\in X:E(x,w)\}|.
\]
Given $\epsilon>0$ and nonempty $X\seq V$ and $Y\seq W$, with $|X|=|Y|$, we say that the pair $(X,Y)$ is \emph{uniformly $\epsilon$-good for $\Gamma$} if either:
\begin{enumerate}[$(i)$]
\item for any $x\in X$ and $y\in Y$, $\deg_\Gamma(x,Y)=\deg_\Gamma(X,y)\leq\epsilon |X|$, or
\item for any $x\in X$ and $y\in Y$, $\deg_\Gamma(x,Y)=\deg_\Gamma(X,y)\geq (1-\epsilon)|X|$.
\end{enumerate}
This choice of terminology is motivated by \cite{MS}, since it is related to the notion of an \emph{$\epsilon$-good} subset in a graph. Next, given $X\seq V$ and $Y\seq W$, let
\[
d_\Gamma(X,Y)=\frac{|E\cap (X\times Y)|}{|X\times Y|}.
\]
Recall that the pair $(X,Y)$ is \emph{$\epsilon$-regular} if, for all $X_0\seq X$ and $Y_0\seq Y$, with $|X_0|\geq \epsilon|X|$ and $|Y_0|\geq\epsilon|Y|$, we have $|d_\Gamma(X,Y)-d_\Gamma(X_0,Y_0)|<\epsilon$. One can show that uniformly $\epsilon^2$-good pairs are $\epsilon$-regular with edge density at most $\epsilon$ or at least $1-\epsilon$ (as in the regularity lemma for stable graphs in \cite{MS}). In fact, regularity occurs between pairs $(X_0,Y_0)$ of nonempty subsets in which only one of $X_0$ or $Y_0$ has size at least $\epsilon|X|$.

\begin{proposition}
Suppose $\Gamma=(V,W;E)$ is a finite bipartite graph and $(X,Y)$ is uniformly $\epsilon^2$-good for $\Gamma$, for some $\epsilon\in (0,\frac{1}{2})$. Then either:
\begin{enumerate}[$(i)$]
\item for all nonempty $X_0\seq X$ and $Y_0\seq Y$, if either $|X_0|\geq \epsilon|X|$ or $|Y_0|\geq \epsilon|Y|$ then $d_\Gamma(X_0,Y_0)\leq \epsilon$, or 
\item for all nonempty $X_0\seq X$ and $Y_0\seq Y$, if either $|X_0|\geq\epsilon|X|$ or $|Y_0|\geq\epsilon|Y|$ then $d_\Gamma(X_0,Y_0)\geq(1-\epsilon)$.
\end{enumerate}
\end{proposition}
\begin{proof}
Assume $(X,Y)$ is uniformly $\epsilon^2$-good for $\Gamma$. We show (i) or (ii) holds (note that $\epsilon<\frac{1}{2}$ prevents both from holding simultaneously). Fix nonempty $X_0\seq X$ and $Y_0\seq Y$ such that $|X_0|\geq \epsilon|X|$ or $|Y_0|\geq \epsilon|Y|$. Without loss of generality, assume $|Y_0|\geq \epsilon|Y|= \epsilon|X|$ (the other case is symmetric). Then
\begin{equation*}
d_\Gamma(X_0,Y_0)=\frac{1}{|X_0\times Y_0|}\sum_{a\in X_0}\deg_\Gamma(x,Y_0).\tag{$\dagger$}
\end{equation*}
If $\deg_\Gamma(a,Y)\leq\epsilon^2|X|$ for all $a\in X$ then $(\dagger)$ implies
\[
d_\Gamma(X_0,Y_0)\leq\frac{\epsilon^2 |X|}{|Y_0|}\leq \epsilon.
\]
On the other hand, if $\deg_\Gamma(a,Y)\geq (1-\epsilon^2)|X|$ for all $a\in X$ then, for all $a\in X$, $\deg_\Gamma(a,Y_0)\geq |Y_0|-\epsilon^2|X|$, and so $(\dagger)$ implies
\[
d_\Gamma(X_0,Y_0)\geq\frac{|Y_0|-\epsilon^2|X|}{|Y_0|}\geq 1-\epsilon.
\]
\end{proof}

The next result gives the graph theoretic regularity statement implied by Theorem \ref{thm:main}. Given a group $G$ and $A\seq G$, we define the \emph{Cayley graph} $C_A(G)$ to be the bipartite graph $(V,W;E)$ where $V=W=G$ and, given $a,b\in G$, $E(a,b)$ holds if and only if $ab\in A$. Let $\deg_A$ denote $\deg_{C_A(G)}$. Note that, given $X\seq G$ and $a\in G$, we have $\deg_A(a,X)=|A\cap aX|$ and $\deg_A(X,a)=|A\cap Xa|$.

\begin{corollary}\label{cor:reg}
For any $k\geq 1$ and $\epsilon>0$, there is $n=n(k,\epsilon)$  such that the following holds. Suppose $G$ is a finite group and $A\seq G$ is $k$-stable. Then there is a partition $C_1,\ldots,C_m$ of $G$ satisfying the following properties.
\begin{enumerate}[$(i)$]
\item $C_1,\ldots,C_m$ are the cosets of a normal subgroup of index $m\leq n$ (so $|C_i|=|C_j|$ for all $i,j\leq m$).
\item For any $i,j\leq m$, $(C_i,C_j)$ is uniformly $\epsilon$-good for $C_A(G)$. 
\end{enumerate}
\end{corollary}
\begin{proof}
Let $n=n(\epsilon,k)$  be given by Theorem \ref{thm:main}. Suppose $G$ is a finite group and $A\seq G$ is $k$-stable. By Theorem \ref{thm:main}, there is a normal subgroup $H$, of index $m\leq n$, such that, for each coset $C$ of $H$, either $|C\cap A|\leq\epsilon |H|$ or $|C\setminus A|\leq\epsilon|H|$. Let $C_1,\ldots,C_m$ be the cosets of $H$.

Fix $i,j\leq m$, and let $C_i=xH$ and $C_j=yH$. Then, using normality of $H$, it follows that for any $a\in C_i$ and $b\in C_j$,
\[
\deg_A(a,C_j)=|A\cap aC_j|=|A\cap xyH|\makebox[.5in]{and} \deg_A(C_i,b)=|A\cap C_jb|=|A\cap Hxy|.
\]
So, by choice of $H$, it follows that $(C_i,C_j)$ is uniformly $\epsilon$-good. 
\end{proof}

\bibliographystyle{amsplain}

\begin{thebibliography}{1}

\bibitem{Green}
B.~Green, \emph{A {S}zemer\'edi-type regularity lemma in abelian groups, with
  applications}, Geom. Funct. Anal. \textbf{15} (2005), no.~2, 340--376.
  \MR{2153903}

\bibitem{H}
Ehud Hrushovski, \emph{Pseudo-finite fields and related structures}, Model
  theory and applications, Quad. Mat., vol.~11, Aracne, Rome, 2002,
  pp.~151--212. \MR{2159717 (2006d:03059)}
  
  \bibitem{HruAG}
Ehud Hrushovski, \emph{Stable group theory and approximate subgroups}, J. Amer.
  Math. Soc. \textbf{25} (2012), no.~1, 189--243. \MR{2833482}

\bibitem{HPP}
Ehud Hrushovski, Ya'acov Peterzil, and Anand Pillay, \emph{Groups, measures,
  and the {NIP}}, J. Amer. Math. Soc. \textbf{21} (2008), no.~2, 563--596.
  \MR{2373360 (2008k:03078)}

\bibitem{HP}
Ehud Hrushovski and Anand Pillay, \emph{Groups definable in local fields and
  pseudo-finite fields}, Israel J. Math. \textbf{85} (1994), no.~1-3, 203--262.
  \MR{1264346}

\bibitem{MP}
Maryanthe Malliaris and Anand Pillay, \emph{The stable regularity lemma
  revisited}, Proc. Amer. Math. Soc. \textbf{144} (2016), no.~4, 1761--1765.
  \MR{3451251}

\bibitem{MS}
M.~Malliaris and S.~Shelah, \emph{Regularity lemmas for stable graphs}, Trans.
  Amer. Math. Soc. \textbf{366} (2014), no.~3, 1551--1585. \MR{3145742}


\bibitem{NP}
Ludomir Newelski and Marcin Petrykowski, \emph{Weak generic types and coverings
  of groups. {I}}, Fund. Math. \textbf{191} (2006), no.~3, 201--225.
  \MR{2278623}
  
  \bibitem{pillay}
Anand Pillay, \emph{Geometric stability theory}, Oxford Logic Guides, vol.~32,
  The Clarendon Press, Oxford University Press, New York, 1996, Oxford Science
  Publications. \MR{1429864}

\bibitem{TW}
C. Terry and J. Wolf, \emph{Stable arithmetic regularity in the
  finite-field model}, arXiv:1710.02021, 2017.

\end{thebibliography}
\end{document}